\documentclass[11pt]{article}
\usepackage{jhtitle,amsbsy,amsmath,amsthm,amssymb,times,graphicx}
\usepackage{probmacros}
\usepackage[sort]{natbib}

\usepackage{comment}
\specialcomment{details}{\begin{quotation}\small}{\end{quotation}}
\excludecomment{details}

\def\baro{\vskip  .2truecm\hfill \hrule height.5pt \vskip  .2truecm}
\def\barba{\vskip -.1truecm\hfill \hrule height.5pt \vskip .4truecm}

\newcommand{\pcite}[1]{\citeauthor{#1}'s \citeyearpar{#1}}

\newtheorem{proposition}{Proposition}
\newtheorem{corollary}{Corollary}
\newtheorem{lemma}{Lemma}

\newtheorem{theorem}{Theorem}
\newtheorem{remark}{Remark}

\newcommand{\X}{{\mathsf{X}}}
\newcommand{\Z}{{\mathsf{Z}}}

\begin{document}

\title{Convergence Analysis of MCMC Algorithms for Bayesian
  Multivariate Linear Regression with Non-Gaussian Errors}
\author{James P. Hobert, Yeun Ji Jung, Kshitij Khare and Qian Qin \\
  Department of Statistics \\ University of Florida} \date{January
  2016}


\abtitle{Convergence of MCMC algorithms}

\amssubj{Primary 60J05; secondary 62F15}

\maketitle

\begin{abstract}
  Gaussian errors are sometimes inappropriate in a multivariate linear
  regression setting because, for example, the data contain outliers.
  In such situations, it is often assumed that the error density is a
  scale mixture of multivariate normal densities that takes the form
  $f(\varepsilon) = \int_0^\infty |\Sigma|^{-\frac{1}{2}}
  u^{\frac{d}{2}} \, \phi_d \big( \Sigma^{-\frac{1}{2}} \sqrt{u} \,
  \varepsilon \big) \, h(u) \, du$, where $d$ is the dimension of the
  response, $\phi_d(\cdot)$ is the standard $d$-variate normal
  density, $\Sigma$ is an unknown $d \times d$ positive definite scale
  matrix, and $h(\cdot)$ is some fixed mixing density.  Combining this
  alternative regression model with a default prior on the unknown
  parameters results in a highly intractable posterior density.
  Fortunately, there is a simple data augmentation (DA) algorithm and
  a corresponding Haar PX-DA algorithm that can be used to explore
  this posterior.  This paper provides conditions (on $h$) for
  geometric ergodicity of the Markov chains underlying these Markov
  chain Monte Carlo (MCMC) algorithms.  These results are extremely
  important from a practical standpoint because geometric ergodicity
  guarantees the existence of the central limit theorems that form the
  basis of all the standard methods of calculating valid asymptotic
  standard errors for MCMC-based estimators.  The main result is that,
  if $h$ converges to 0 at the origin at an appropriate rate, and
  $\int_0^\infty u^{\frac{d}{2}} \, h(u) \, du < \infty$, then the DA
  and Haar PX-DA Markov chains are both geometrically ergodic.  This
  result is quite far-reaching.  For example, it implies the geometric
  ergodicity of the DA and Haar PX-DA Markov chains whenever $h$ is
  generalized inverse Gaussian, log-normal, inverted gamma (with shape
  parameter larger than $d/2$), or Fr\'{e}chet (with shape parameter
  larger than $d/2$).  The result also applies to certain subsets of
  the gamma, $F$, and Weibull families.
\end{abstract}

\section{Introduction}
  \label{sec:intro}
  Let $Y_1,Y_2,\dots,Y_n$ be independent $d$-dimensional random
  vectors from the multivariate linear regression model
\begin{equation}
  \label{eq:mreg}
  Y_i = \beta^T x_i   + \Sigma^{\frac{1}{2}} \varepsilon_i \;,
\end{equation}
where $x_i$ is a $p \times 1$ vector of known covariates associated
with $Y_i$, $\beta$ is a $p \times d$ matrix of unknown regression
coefficients, $\Sigma$ is an unknown positive definite scale matrix,
and $\varepsilon_1,\dots,\varepsilon_n$ are iid errors.  In situations
where Gaussian errors are inappropriate, e.g., when the data contain
outliers, scale mixtures of multivariate normal densities constitute a
rich class of alternative error densities \citep[see,
e.g.,][]{andr:mall:1974,fern:stee:1999,fern:stee:2000,west:1984}.
These mixtures take the form
\begin{equation*}
  f_h(\varepsilon) = \int_{0}^{\infty} \frac{u^{\frac{d}{2}}}{(2
    \pi)^{\frac{d}{2}}} \, \exp \Big \{ -\frac{u}{2}
  \varepsilon^T \varepsilon \Big \} h(u) \, du \;,
\end{equation*}
where $h$ is the density function of some positive random variable.
We shall refer to $h$ as a \textit{mixing density}.  By varying the
mixing density, one can construct error densities with many different
types of tail behavior.  A well-known example is that when $h$ is the
density of a $\mbox{Gamma}(\frac{\nu}{2}, \frac{\nu}{2})$ random
variable, then $f_h$ becomes the multivariate Student's $t$ density
with $\nu$ degrees of freedom, which, aside from a normalizing
constant, is given by $\big[ 1 + {\nu}^{-1} \varepsilon^T
\varepsilon\big]^{-\frac{d + \nu}{2}}$.

Let $Y$ denote the $n \times d$ matrix whose $i$th row is $Y_i^T$, and
let $X$ stand for the $n \times p$ matrix whose $i$th row is $x_i^T$,
and, finally, let $\varepsilon$ represent the $n \times d$ matrix
whose $i$th row is $\varepsilon_i^T$.  Using this notation, we can
state the $n$ equations in \eqref{eq:mreg} more succinctly as follows
\begin{equation}
  \label{eq:mreg2}
  Y = X \beta + \varepsilon \, \Sigma^{\frac{1}{2}} \;.
\end{equation}
Let $y$ and $y_i$ denote the observed values of $Y$ and $Y_i$,
respectively.

Consider a Bayesian analysis of the data from the regression model
\eqref{eq:mreg2} using an improper prior on $(\beta,\Sigma)$ that
takes the form $\omega (\beta , \Sigma) \propto |\Sigma|^{-a} \,
I_{{\cal S}_d}(\Sigma)$ where ${\cal S}_d \subset
\mathbb{R}^{\frac{d(d + 1)}{2}}$ denotes the space of $d \times d$
positive definite matrices.  Taking $a=(d+1)/2$ yields the
independence Jeffreys prior, which is a standard default prior for
multivariate location scale problems.  The joint density of the data
from model \eqref{eq:mreg2} is, of course, given by
\begin{equation}
  \label{eq:joint}
  f (y | \beta, \Sigma) = \prod_{i=1}^n \Bigg[ \int_{0}^{\infty}
  \frac{u^{\frac{d}{2}}}{(2\pi)^{\frac{d}{2}} |\Sigma|^{\frac{1}{2}}}
  \exp \bigg\{ -\frac{u}{2} \Big(y_i - \beta^T x_i\Big)^T
  \Sigma^{-1}\Big(y_i - \beta^T x_i\Big) \bigg\} h(u) \, du \Bigg] \;.
\end{equation}
Define
\[
m(y) = \int_{{\cal S}_d} \int_{\mathbb{R}^{p \times d}}
f(y|\beta,\Sigma) \, \omega(\beta,\Sigma) \, d\beta \, d\Sigma \;.
\]
The posterior distribution is proper precisely when $m(y) < \infty$.
Let $\Lambda$ denote the $n \times (p+d)$ matrix $(X:y)$.  As we shall
see, the following conditions are necessary for propriety:
\begin{enumerate}
  \item[($N1$)] $\mbox{rank}(\Lambda) = p+d \;;$
  \item[($N2$)] $n > p + 2d - 2a \;.$
\end{enumerate}
We assume throughout the paper that $(N1)$ and $(N2)$ hold.  Under
these two conditions, the Markov chain of interest is well-defined,
and we can engage in a convergence rate analysis whether the posterior
is proper or not.  This is a subtle point upon which we will expand in
Section~\ref{sec:d_m}.

Of course, when the posterior is proper, it is given by
\[
\pi^*(\beta, \Sigma|y) = \frac{f(y|\beta,\Sigma) \,
  \omega(\beta,\Sigma)}{m(y)} \;.
\]
This density is (nearly always) intractable in the sense that
posterior expectations cannot be computed in closed form.  However,
there is a well-known data augmentation algorithm (or two-variable
Gibbs sampler) that can be used to explore this intractable posterior
density \citep[see, e.g.,][]{liu:1996}.  In order to state this
algorithm, we must introduce some additional notation.  For $z =
(z_1,\dots,z_n)$, let $Q$ be an $n \times n$ diagonal matrix whose
$i$th diagonal element is $z_i^{-1}$.  Also, define $\Omega = (X^T
Q^{-1} X)^{-1}$ and $\mu = (X^T Q^{-1} X)^{-1} X^T Q^{-1} y$.  We
shall assume throughout the paper that
\[
\int_0^\infty u^{\frac{d}{2}} \, h(u) \, du < \infty \;,
\]
where $h$ is the mixing density, and we will refer to this condition
as ``condition ${\cal M}$.''  Finally, define a parametric family of
univariate density functions indexed by $s \ge 0$ as follows
\begin{equation*}
  \psi(u;s) = b(s) \, u^{\frac{d}{2}} \, e^{ -\frac{s u}{2}} \, h(u) \;,
\end{equation*}
where $b(s)$ is the normalizing constant.  The data augmentation (DA)
algorithm calls for draws from the inverse Wishart ($\mbox{IW}_d$) and
matrix normal ($\mbox{N}_{p,d}$) distributions.  The precise forms of
the densities are given in the Appendix.  We now present the DA
algorithm.  If the current state of the DA Markov chain is
$(\beta_m,\Sigma_m) = (\beta,\Sigma)$, then we simulate the new state,
$(\beta_{m+1},\Sigma_{m+1})$, using the following three-step
procedure.

\baro \vspace*{2mm}
\noindent {\rm Iteration $m+1$ of the DA algorithm:}
\begin{enumerate}
\item Draw $\{Z_i\}_{i=1}^n$ independently with $Z_i \sim \psi \Big(
  \cdot \;; \big( \beta^T x_i - y_i \big)^T \Sigma^{-1} \big( \beta^T
  x_i - y_i \big) \Big)$, and call the result $z = (z_1,\dots,z_n)$.
\item Draw
\[
\Sigma_{m+1} \sim \mbox{IW}_d \bigg( n-p+2a-d-1, \Big( y^T Q^{-1} y -
  \mu^T \Omega^{-1} \mu \Big)^{-1} \bigg)  \;.
\]
\item Draw $\beta_{m+1} \sim \mbox{N}_{p,d} \big( \mu, \Omega,
  \Sigma_{m+1} \big)$ \vspace*{-2.5mm}
\end{enumerate}
\barba

Obviously, in order to run this algorithm, one must be able to make
draws from $\psi(\cdot \, ;s)$.  When $h$ is a standard density,
$\psi$ often turns out to be one as well.  For example, when $h$ is a
gamma density, $\psi$ is also gamma, and when $h$ is inverted gamma,
$\psi$ is generalized inverse Gaussian (see Section~\ref{sec:ex}).
Even when $\psi$ is not a standard density, it is still a simple
entity - a univariate density on $(0,\infty)$ - and so is usually
amenable to straightforward sampling.  In particular, if it is
possible to make draws from $h$, then $h$ can be used as the candidate
in a simple rejection sampler for $\psi$.

Denote the DA Markov chain by $\Phi =
\{(\beta_m,\Sigma_m)\}_{m=0}^\infty$.  The main contribution of this
paper is to demonstrate that $\Phi$ is geometrically ergodic as long
as $h$ converges to zero at the origin at an appropriate rate.  (A
formal definition of geometric ergodicity is given in
Section~\ref{sec:d_m}.)  Our result is remarkable both for its
simplicity and for its scope.  Indeed, the conditions turn out to be
extremely simple to check, and, at the same time, the result applies
to a huge class of Monte Carlo Markov chains.  It is well known among
Markov chain Monte Carlo (MCMC) experts that establishing geometric
ergodicity of practically relevant chains is extremely challenging.
Thus, it is noteworthy that we are able to handle so many such chains
simultaneously.  Of course, the important practical and theoretical
benefits of basing one's MCMC algorithm on a geometrically ergodic
Markov chain have been well-documented by, e.g.,
\citet{robe:rose:1998}, \citet{jone:hobe:2001} and
\citet{fleg:hara:jone:2008}.  In order to give a precise statement of
our main result, we now define three classes of mixing densities based
on behavior near the origin.

Define $\mathbb{R}_+ = (0,\infty)$, and let $h: \mathbb{R}_+
\rightarrow [0,\infty)$ be a mixing density.  If there is a $\delta>0$
such that $h(u)=0$ for all $u \in (0,\delta)$, then we say that $h$ is
\textit{zero near the origin}.  Now assume that $h$ is strictly
positive in a neighborhood of 0 (i.e., $h$ is not zero near the
origin).  If there exists a $c>-1$ such that
\[
\lim_{u \rightarrow 0} \frac{h(u)}{u^c} \in \mathbb{R}_+ \;,
\]
then we say that $h$ is \textit{polynomial near the origin} with power
$c$.  Finally, if for every $c>0$, there exists an $\eta_c>0$ such
that the ratio $\frac{h(u)}{u^c}$ is strictly increasing in
$(0,\eta_c)$, then we say that $h$ is \textit{faster than polynomial
  near the origin}.

Every mixing density that is a member of a standard parametric family
is either polynomial near the origin, or faster than polynomial near
the origin.  Indeed, the gamma, beta, $F$, Weibull, and shifted Pareto
densities are all polynomial near the origin, whereas the inverted
gamma, log-normal, generalized inverse Gaussian, and Fr\'{e}chet
densities are all faster than polynomial near the origin.  We
establish these facts in Section~\ref{sec:ex}.  Here is our main
result.

\begin{theorem}
  \label{thm:main}
  Let $h$ be a mixing density that satisfies condition ${\cal M}$.
  Assume that $h$ is zero near the origin, or faster than polynomial
  near the origin, or polynomial near the origin with power $c >
  \frac{n-p+2a-d-1}{2}$.  Then the posterior distribution is proper
  and the DA Markov chain is geometrically ergodic.
\end{theorem}

This result is more substantial than typical convergence rate results
for DA algorithms and Gibbs samplers in the sense that it applies to a
huge class of mixing densities, whereas typical results apply to
relatively small parametric families of Markov chains \citep[see,
e.g.,][]{pal:khar:2014}.  Note that, outside of the polynomial case,
the only regularity condition in Theorem~\ref{thm:main} is the rather
weak requirement that $\int_0^\infty u^{\frac{d}{2}} \, h(u) \, du <
\infty$.  Thus, for example, Theorem~\ref{thm:main} implies that if
$h$ is generalized inverse Gaussian, log-normal, inverted gamma (with
shape parameter larger than $d/2$), or Fr\'{e}chet (with shape
parameter larger than $d/2$), then the DA Markov chain converges at a
geometric rate.  

Another notable consequence of Theorem 1 is the following.  Suppose
that $h$ satisfies the conditions of Theorem~\ref{thm:main}, and let
$B>0$.  Note that we can alter $h$ on the set $[B,\infty)$ in any way
we like, and, as long as condition ${\cal M}$ continues to hold, the
corresponding Markov chain will still be geometrically ergodic.

When $h$ is polynomial near the origin, there is an extra regularity
condition for geometric ergodicity that can be somewhat restrictive.
For example, take the case where $h$ is the gamma density with shape
and rate both equal to $\nu/2$ (so the error density is Student's $t$
with $\nu$ degrees of freedom).  In this case, Theorem~\ref{thm:main}
implies that the DA Markov chain will converge at a geometric rate as
long as $\nu > n-p+2a-d+1$.  If $n-p+2a-d+1$ is small, then this
condition is not too troublesome.  However, if this number happens to
be large, then Theorem~\ref{thm:main} applies only when the degrees of
freedom of the $t$ distribution are large, which is not very useful.
It is an open question whether the condition $c >
\frac{n-p+2a-d-1}{2}$ is necessary.

A couple of special cases of Theorem~\ref{thm:main} have appeared
previously in the literature.  In particular, the result for the gamma
mixing density described above was established by
\citet{roy:hobe:2010} in the special case of the independence Jeffreys
prior where $a = (d+1)/2$.  Also, \citet{jung:hobe:2014} showed that,
when $d=1$ and the mixing density is inverted gamma with shape
parameter larger than 1/2, the Markov operator associated with the DA
Markov chain is a \textit{trace-class} operator, which implies that
the corresponding chain converges at a geometric rate.

It is often possible to convert a DA algorithm into a Haar PX-DA
algorithm that is theoretically superior to the underlying DA
algorithm, yet essentially equivalent in terms of simulation effort
\citep[see, e.g.,][]{liu:wu:1999,hobe:marc:2008}.  In fact,
\citet{roy:hobe:2010} developed a Haar PX-DA variant of the DA
algorithm described above for the special case in which
$a=\frac{d+1}{2}$.  It turns out that, when $a \ne \frac{d+1}{2}$, an
additional regularity condition on $h$ is required in order to define
this alternative algorithm.  In particular, the Haar PX-DA algorithm
can be defined only when
\begin{equation}
  \label{eq:H}
  \int_0^\infty t^{n+\frac{(d+1-2a)d}{2}-1} \, \Bigg[ \prod_{i=1}^n h(t z_i) \Bigg] \,
  dt < \infty
\end{equation}
for (almost) all $z \in \mathbb{R}_+^n$.  An argument similar to one
in used \citet[][Section 3]{roy:hobe:2010} shows that \eqref{eq:H}
holds if
\begin{equation}
  \label{eq:Ha}
\int_0^\infty u^{\frac{(d+1-2a)d}{2}} \, h(u) \, du < \infty \;.
\end{equation}
Note that \eqref{eq:Ha} \textit{always} holds when $a =
\frac{d+1}{2}$.  Now assume that \eqref{eq:H} holds, and define a
parametric family of density functions, indexed by $z \in
\mathbb{R}_+^n$, that take the form
\[
\xi(v;z) \propto v^{n+\frac{(d+1-2a)d}{2}-1} \, \Bigg[ \prod_{i=1}^n
h(v z_i) \Bigg] \, I_{\mathbb{R}_+}(v) \;.
\]
As with the parametric family $\psi(\cdot\,;s)$, when $h$ is a
standard density, $\xi$ often turns out to be standard as well.  For
example, if $h$ is gamma, inverted gamma, or generalized inverse
Gaussian, then $\xi$ turns out to be a member of the same parametric
family.  If the current state of the Haar PX-DA Markov chain is
$(\beta^*_m,\Sigma^*_m) = (\beta,\Sigma)$, then we simulate the new
state, $(\beta^*_{m+1},\Sigma^*_{m+1})$, using the following four-step
procedure.

\baro \vspace*{2mm}
\noindent {\rm Iteration $m+1$ of the Haar PX-DA algorithm:}
\begin{enumerate}
\item Draw $\{Z'_i\}_{i=1}^n$ independently with $Z'_i \sim \psi \Big(
  \cdot \;; \big( \beta^T x_i - y_i \big)^T \Sigma^{-1} \big( \beta^T x_i -
  y_i \big) \Big)$, and call the result $z' = (z'_1,\dots,z'_n)$.
\item Draw $V \sim \xi(\cdot \, ;z')$, call the result $v$, and set $z
  = (vz'_1,\dots,vz'_n)^T$.
\item Draw
\[
\Sigma^*_{m+1} \sim \mbox{IW}_d \bigg( n-p+2a-d-1, \Big( y^T Q^{-1} y -
  \mu^T \Omega^{-1} \mu \Big)^{-1} \bigg)  \;.
\]
\item Draw $\beta^*_{m+1} \sim \mbox{N}_{p,d} \big( \mu, \Omega,
  \Sigma^*_{m+1} \big)$ \vspace*{-2.5mm}
\end{enumerate}
\barba

Note that the only difference between this algorithm and the DA
algorithm is one extra univariate draw (from $\xi(\cdot \, ; \,
\cdot)$) per iteration.  Hence, the two algorithms are virtually
equivalent from a computational standpoint.  Theoretically, the Haar
PX-DA algorithm is at least as good as the DA algorithm, both in terms
of convergence rate (operator norm) and asymptotic efficiency
\citep{liu:wu:1999,hobe:marc:2008,khar:hobe:2011}.  Moreover, there is
a great deal of empirical evidence that the Haar PX-DA algorithm can
be far superior \citep[see, e.g.][]{meng:vand:1999,vand:meng:2001}.
The following corollary to Theorem~\ref{thm:main} is an immediate
consequence of the fact that, in general, the norm of the Markov
operator of a Haar PX-DA chain is no larger than that of the
underlying DA chain.

\begin{corollary}
  \label{cor:cor_main} 
  Let $h$ be a mixing density that satisfies condition ${\cal M}$, and
  assume that \eqref{eq:H} holds.  Assume that $h$ is zero near the
  origin, or faster than polynomial near the origin, or polynomial
  near the origin with power $c > \frac{n-p+2a-d-1}{2}$.  Then the
  Haar PX-DA Markov chain is geometrically ergodic.
\end{corollary}

The remainder of this paper is organized as follows.
Section~\ref{sec:DA} contains a brief description of the latent data
model that leads to the DA algorithm, as well as a formal definition
of the DA Markov chain.  Section~\ref{sec:d_m} contains a drift and
minorization analysis of $\Phi$ that culminates in a simple sufficient
condition for geometric ergodicity that depends only on $h$.  This
result is used to prove Theorem~\ref{thm:main} in
Section~\ref{sec:proof}.  In Section~\ref{sec:ex}, we consider the
implications of Theorem~\ref{thm:main} when $h$ is a member of one of
the standard parametric families, and we also develop conditions under
which a mixture of mixing densities leads to a geometric DA Markov
chain.  Finally, the Appendix contains the definitions of the inverse
Wishart ($\mbox{IW}_d$) and matrix normal ($\mbox{N}_{p,d}$)
densities.

\section{The latent data model and the DA Markov chain}
\label{sec:DA}

In order to formally define the Markov chain that the DA algorithm
simulates, we must introduce the latent data model.  Suppose that,
conditional on $(\beta,\Sigma)$, $\{(Y_i,Z_i)\}_{i=1}^n$ are iid pairs
such that
\[
Y_i| Z_i=z_i \sim \mbox{N}_d \big( \beta^T x_i, \Sigma
/ z_i \big)
\]
\[
Z_i \sim h \;.
\]
Denote the joint density of $\{(Y_i,Z_i)\}_{i=1}^n$ by $\tilde{f}(y, z
\big| \beta, \Sigma)$.  It's easy to see that
\begin{equation*}
  \int_{\mathbb{R}_+^n} \tilde{f}(y, z \big| \beta,
  \Sigma) \, dz = f(y \big| \beta, \Sigma) \;,
\end{equation*}
where the right-hand side is the joint density of the data defined at
\eqref{eq:joint}.  Now define a (possibly improper) density on
$\mathbb{R}^{p \times d} \times {\cal S}_d \times \mathbb{R}_+^n$ as
follows
\[
\pi(\beta, \Sigma,z \big| y) = \tilde{f}(y, z \big| \beta,
\Sigma) \, \omega (\beta , \Sigma) \;,
\]
and note that
\begin{equation}
  \label{eq:invariant_density}
  \int_{\mathbb{R}^n_+} \pi(\beta,\Sigma,z|y) \, dz = f(y \big|
  \beta, \Sigma) \, \omega (\beta , \Sigma) \;.
\end{equation}
It follows that $\pi(\beta, \Sigma,z \big| y)$ is a proper density if
and only if the posterior distribution is proper.  Importantly,
whether $\pi(\beta, \Sigma,z \big| y)$ is proper or not, conditions
($N1$) and ($N2$) guarantee that the corresponding ``conditional''
densities, $\pi(\beta,\Sigma|z,y)$ and $\pi(z|\beta,\Sigma,y)$, are
well-defined.  Indeed, $\pi(\beta,\Sigma|z,y) = \pi(\beta|\Sigma,z,y)
\pi(\Sigma|z,y)$, and routine calculations show that
$\pi(\beta|\Sigma,z,y)$ is a matrix normal density, and
$\pi(\Sigma|z,y)$ is an inverse Wishart density.  (The precise forms
of these densities can be gleaned from the algorithm stated in the
Introduction.)  It is also straightforward to show that
\[
\pi(z|\beta,\Sigma,y) = \prod_{i=1}^n \psi(z_i;r_i) \;,
\]
where $r_i = \big( \beta^T x_i - y_i \big)^T \Sigma^{-1} \big( \beta^T
x_i - y_i \big)$ for $i=1,2,\dots,n$.

The DA algorithm simulates the Markov chain $\Phi =
\{(\beta_m,\Sigma_m)\}_{m=0}^\infty$, whose state space is $\X :=
\mathbb{R}^{p \times d} \times {\cal S}_d$, and whose Markov
transition density (Mtd)
\begin{equation*}
  k \big( \beta,\Sigma \big| \tilde{\beta},\tilde{\Sigma} \big) =
  \int_{\mathbb{R}^n_+} \pi(\beta,\Sigma|z,y) \,
  \pi(z|\tilde{\beta},\tilde{\Sigma},y) \, dz \;.
\end{equation*}
We suppress dependence on the data, $y$, since it is fixed throughout.
Note that $\pi(\beta,\Sigma|z,y)$ and $\pi(z|\beta,\Sigma,y)$ are both
strictly positive on $\Z = \{z \in \mathbb{R}_+ : h(z) > 0\}$, and
$\Z$ has positive Lebesgue measure.  Therefore, $k \big( \beta,\Sigma
\big| \tilde{\beta},\tilde{\Sigma} \big)$ is strictly positive on $\X
\times \X$, which implies irreducibility and aperiodicity.  It's easy
to see that \eqref{eq:invariant_density} is an invariant density for
$\Phi$.  Consequently, if the posterior is proper, then the chain's
invariant density is the target posterior, $\pi^*(\beta,\Sigma|y)$,
and the chain is positive recurrent.  In fact, it is positive Harris
recurrent (because $k$ is strictly positive).

We end this section by describing an interesting simplification that
occurs in the special case where $a=(d+1)/2$ and $n=p+d$.
\citet{roy:hobe:2010} show that when $a=(d+1)/2$, we have
\[
\pi(z|y) = \int_{{\cal S}_d} \int_{\mathbb{R}^{p \times d}}
\pi(\beta,\Sigma,z|y) \, d\beta \, d\Sigma \propto \frac{\prod_{i=1}^n
  h(z_i)}{|Q|^{\frac{d}{2}} |\Omega|^{\frac{n-p-d}{2}} | \Lambda^T
  Q^{-1} \Lambda|^{\frac{n-p}{2}}} \;,
\]
which is not necessarily integrable in $z$, because the posterior is
not necessarily proper \citep[see, e.g.,][]{fern:stee:1999}.  However,
when $n=p+d$, $\Lambda$ is square and non-singular (because of
$(N1)$), and we have the stunningly simple formula
\[
\pi(z|y) \propto \prod_{i=1}^n h(z_i)  \;.
\]
Consequently, when $a=(d+1)/2$ and $n=p+d$, the posterior distribution
is proper, and if we are able to draw from the mixing density, $h$,
then we can make an exact draw from the posterior density by drawing
sequentially from $\pi(z|y)$, $\pi(\Sigma|z,y)$, and
$\pi(\beta|\Sigma,z,y)$, and then ignoring $z$.

In the next section, we develop a condition on $h$ that implies
geometric ergodicity of the DA Markov chain, $\Phi$.

\section{A Drift and Minorization Analysis of $\Phi$}
\label{sec:d_m}

Here we analyze the DA Markov chain via drift and minorization
arguments.  For background on these techniques, see
\citet{jone:hobe:2001} and \citet{robe:rose:2004}.  Suppose that the
posterior distribution is proper.  Then the DA Markov chain $\Phi$ is
\textit{geometrically ergodic} if there exist $M: \X \rightarrow
[0,\infty)$ and $\rho \in [0,1)$ such that, for all $m \in
\mathbb{N}$,
\begin{equation}
  \label{eq:ge}
  \int_{{\cal S}_d} \int_{\mathbb{R}^{p \times d}} \Big| k^m \big(
  \beta,\Sigma \big| \tilde{\beta},\tilde{\Sigma} \big) -
  \pi^*(\beta,\Sigma|y) \Big| \, d\beta \, d\Sigma \le
  M(\tilde{\beta},\tilde{\Sigma}) \, \rho^m \;,
\end{equation}
where $k^m$ is the $m$-step Mtd.  The quantity on the left-hand side
of \eqref{eq:ge} is, of course, the total variation distance between
the posterior distribution and the distribution of
$(\beta_m,\Sigma_m)$ conditional on $(\beta_0,\Sigma_0) =
(\tilde{\beta},\tilde{\Sigma})$.  Here is the main result of this
section.

\begin{proposition}
  \label{prop:drift_smn}
  Let $h$ be a mixing density that satisfies condition ${\cal M}$.
  Suppose that there exist $\lambda \in \big[ 0, \frac{1}{n-p+2a-1}
  \big)$ and $L \in \mathbb{R}$ such that
\begin{equation}
   \label{eq:key}
   \frac{\int_0^\infty u^{\frac{d-2}{2}} \, e^{ -\frac{s u}{2}} \, h(u)
     \, du}{\int_0^\infty u^{\frac{d}{2}} \, e^{ -\frac{s u}{2}} \, h(u)
     \, du} \le \lambda s + L
\end{equation}
for every $s \ge 0$.  Then the posterior distribution is proper, and
the DA Markov chain is geometrically ergodic.
\end{proposition}

\begin{proof}
  We will prove the result by establishing a drift condition and an
  associated minorization condition, as in \pcite{rose:1995} Theorem
  12.  We begin by noting that the drift and minorization technique is
  applicable whether the posterior distribution is proper or not.  (In
  more technical terms, it is not necessary to demonstrate that the
  Markov chain under study is positive recurrent before applying the
  technique.)  Moreover, the DA Markov chain cannot be geometrically
  ergodic if the posterior is improper (since the corresponding chain
  is not positive recurrent).  Hence, conditions that imply geometric
  ergodicity of the DA Markov chain simultaneously imply propriety of
  the corresponding posterior distribution.

  Our drift function, $V: \mathbb{R}^{p \times d} \times {\cal S}_d
  \rightarrow \mathbb{R}_+$, is as follows
\begin{equation*}
  V(\beta,\Sigma) = \sum_{i=1}^n \big(y_i - \beta^T x_i)^T \Sigma^{-1}
  \big(y_i - \beta^T x_i) \;.
\end{equation*}

\noindent {\bf Part I: Minorization}.  Fix $l>0$ and define
\[
B_l = \big \{ (\beta,\Sigma) : V(\beta,\Sigma) \le l \big \} \;.
\]
We will construct $\epsilon \in (0,1)$ and a density function $f^*:
\mathbb{R}^{p \times d} \times {\cal S}_d \rightarrow [0,\infty)$
(both of which depend on $l$) such that, for all
$(\tilde{\beta},\tilde{\Sigma}) \in B_l$,
\[
k(\beta,\Sigma | \tilde{\beta},\tilde{\Sigma}) \ge \epsilon
f^*(\beta,\Sigma) \;.
\]
This is the minorization condition.  We note that it suffices to
construct $\epsilon \in (0,1)$ and a density function $\hat{f}:
\mathbb{R}_+^n \rightarrow [0,\infty)$ such that, for all
$(\tilde{\beta},\tilde{\Sigma}) \in B_l$,
\[
\pi(z|\tilde{\beta},\tilde{\Sigma},y) \ge \epsilon \hat{f}(z) \;.
\]
Indeed, if such an $\hat{f}$ exists, then for all
$(\tilde{\beta},\tilde{\Sigma}) \in B_l$, we have
\begin{equation*}
  k \big( \beta,\Sigma \big| \tilde{\beta},\tilde{\Sigma} \big) =
  \int_{\mathbb{R}^n_+} \pi(\beta,\Sigma|z,y) \,
  \pi(z|\tilde{\beta},\tilde{\Sigma},y) \, dz \ge \epsilon
  \int_{\mathbb{R}^n_+} \pi(\beta,\Sigma|z,y) \, \hat{f}(z) \, dz
  = \epsilon f^*(\beta,\Sigma) \;.
\end{equation*}
We now build $\hat{f}$.  Define $\tilde{r}_i = \big(y_i -
\tilde{\beta}^T x_i)^T \tilde{\Sigma}^{-1} \big(y_i - \tilde{\beta}^T
x_i)$, and note that
\begin{equation*}
  \pi(z|\tilde{\beta},\tilde{\Sigma},y) = \prod_{i=1}^n
  \psi(z_i;\tilde{r}_i) = \prod_{i=1}^n b(\tilde{r}_i) \,
  z^{\frac{d}{2}}_i \, e^{ -\frac{\tilde{r}_i z_i}{2}} \, h(z_i) \;.
\end{equation*}
Now, for any $s \ge 0$, we have
\[
b(s) = \frac{1}{\int_0^\infty u^{\frac{d}{2}} \, e^{ -\frac{s u}{2}}
  \, h(u) \, du} \ge \frac{1}{\int_0^\infty u^{\frac{d}{2}} \, h(u) \,
  du} \;.
\]
By definition, if $(\tilde{\beta},\tilde{\Sigma}) \in B_l$, then
$\sum_{i=1}^n \tilde{r}_i \le l$, which implies that $\tilde{r}_i \le
l$ for each $i=1,\dots,n$.  Thus, if $(\tilde{\beta},\tilde{\Sigma})
\in B_l$, then for each $i=1,\dots,n$, we have
\[
z^{\frac{d}{2}}_i \, e^{ -\frac{\tilde{r}_i z_i}{2}} \, h(z_i) \ge
z^{\frac{d}{2}}_i \, e^{ -\frac{l z_i}{2}} \, h(z_i) \;.
\]
Therefore,
\begin{align*}
  \pi(z|\tilde{\beta},\tilde{\Sigma},y) & \ge \bigg[ \int_0^\infty
  u^{\frac{d}{2}} \, h(u) \, du \bigg]^{-n} \prod_{i=1}^n
  z^{\frac{d}{2}}_i \, e^{- \frac{l z_i}{2}} \, h(z_i) \\ & = \bigg[
  \frac{\int_0^\infty u^{\frac{d}{2}} \, e^{- \frac{l u}{2}} \, h(u)
    \, du}{\int_0^\infty u^{\frac{d}{2}} \, h(u) \, du} \bigg]^n
  \prod_{i=1}^n \frac{z^{\frac{d}{2}}_i \, e^{- \frac{l z_i}{2}} \,
    h(z_i)}{\int_0^\infty u^{\frac{d}{2}} \, e^{- \frac{l u}{2}} \,
    h(u) \, du} \\ & := \epsilon \hat{f}(z) \;.
\end{align*}
Hence, our minorization condition is established.

\noindent {\bf Part II: Drift}.  To establish the required drift
condition, we need to bound the expectation of
$V(\beta_{m+1},\Sigma_{m+1})$ given that $(\beta_m,\Sigma_m) =
(\tilde{\beta},\tilde{\Sigma})$.  This expectation is given by
\begin{align*}
  \int_{{\cal S}_d} \int_{\mathbb{R}^{p \times d}} & V(\beta,\Sigma)
  \, k(\beta,\Sigma | \tilde{\beta},\tilde{\Sigma}) \, d\beta \,
  d\Sigma
  \nonumber \\
  & = \int_{\mathbb{R}_+^n} \Bigg \{ \int_{{\cal S}_d} \bigg[
  \int_{\mathbb{R}^{p \times d}} V(\beta,\Sigma) \,
  \pi(\beta|\Sigma,z,y) \, d\beta \bigg] \, \pi(\Sigma|z,y) \, d\Sigma
  \Bigg \} \pi(z|\tilde{\beta},\tilde{\Sigma},y) \, dz \;.
\end{align*}
Calculations in \pcite{roy:hobe:2010} Section 4 show that
\begin{equation*}
  \int_{{\cal S}_d} \bigg[ \int_{\mathbb{R}^{p \times d}}
  V(\beta,\Sigma) \, \pi(\beta|\Sigma,z,y) \, d\beta \bigg] \,
  \pi(\Sigma|z,y) \, d\Sigma \le (n-p+2a-1) \sum_{i=1}^n \frac{1}{z_i} \;.
\end{equation*}
It follows from \eqref{eq:key} that
\begin{align*}
  \int_{\mathbb{R}_+^n} \Bigg \{ \int_{{\cal S}_d} \bigg[
  \int_{\mathbb{R}^{p \times d}} V(\beta,\Sigma) \,
  \pi(\beta|\Sigma,z,y) & \, d\beta \bigg] \, \pi(\Sigma|z,y) \, d\Sigma
  \Bigg \} \pi(z|\tilde{\beta},\tilde{\Sigma},y) \, dz \\ & \le
  (n-p+2a-1) \int_{\mathbb{R}_+^n} \bigg[ \sum_{i=1}^n \frac{1}{z_i}
  \bigg] \pi(z|\tilde{\beta},\tilde{\sigma},y) \, dz \\ & = (n-p+2a-1)
  \sum_{i=1}^n b(\tilde{r}_i) \int_0^\infty
  u^{\frac{d-2}{2}} \, e^{ -\frac{\tilde{r}_i u}{2}} \, h(u) \, du \\
  & \le (n-p+2a-1) \bigg(\lambda \sum_{i=1}^n \tilde{r}_i + nL \bigg)
  \\ & = \lambda (n-p+2a-1) V(\tilde{\beta},\tilde{\Sigma}) +
  (n-p+2a-1) n L \\ & = \lambda' V(\tilde{\beta},\tilde{\sigma}) + L'
  \;,
\end{align*}
where $\lambda' := \lambda (n-p+2a-1) \in [0,1)$ and $L' := (n-p+2a-1)
n L$.  Since the minorization condition holds for any $l>0$, an appeal
to \pcite{rose:1995} Theorem 12 yields the result.  This completes the
proof.
\end{proof}

\begin{remark}
  A straightforward argument shows that, if the mixing density $h(u)$
  satisfies the conditions of Proposition~\ref{prop:drift_smn}, then
  so does every member of the corresponding scale family given by
  $\frac{1}{\sigma} h \big( \frac{u}{\sigma} \big)$, for $\sigma>0$.
\end{remark}

In the next section, we parlay Proposition~\ref{prop:drift_smn} into a
proof of Theorem~\ref{thm:main}.  The key is to show that $h$
satisfies \eqref{eq:key} as long as it converges to zero at the origin
at an appropriate rate.

\section{Proof of Theorem~\ref{thm:main}}
\label{sec:proof}

In this section, we prove three corollaries, which, taken together,
constitute Theorem~\ref{thm:main}.  There is one corollary for each of
the three classes of mixing densities defined in the Introduction.

\subsection{Case I: Zero near the origin}
\label{sec:zno}

\begin{corollary}
  \label{prop:zno}
  Let $h$ be a mixing density that satisfies condition ${\cal M}$.  If
  $h$ is zero near the origin, then the posterior distribution is
  proper and the DA Markov chain is geometrically ergodic.
\end{corollary}

\begin{proof}
  Fix $s \ge 0$, and recall that $h(u)=0$ for $u \in (0,\delta)$ for
  some $\delta>0$.  Hence,
\[
  \frac{\int_0^\infty u^{\frac{d-2}{2}} \, e^{ -\frac{s u}{2}} \, h(u)
    \, du}{\int_0^\infty u^{\frac{d}{2}} \, e^{ -\frac{s u}{2}} \,
    h(u) \, du} = \frac{\int_\delta^\infty \frac{1}{\sqrt{u}}
    u^{\frac{d-1}{2}} \, e^{ -\frac{s u}{2}} \, h(u) \,
    du}{\int_\delta^\infty \sqrt{u} \, u^{\frac{d-1}{2}} \, e^{
      -\frac{s u}{2}} \, h(u) \, du} \le
  \frac{\frac{1}{\sqrt{\delta}} \int_\delta^\infty u^{\frac{d-1}{2}}
    \, e^{ -\frac{s u}{2}} \, h(u) \, du}{\sqrt{\delta}
    \int_\delta^\infty u^{\frac{d-1}{2}} \, e^{ -\frac{s u}{2}} \,
    h(u) \, du} = \frac{1}{\delta} \;.
\]
Thus, the conditions of Proposition~\ref{prop:drift_smn} are satisfied
and the proof is complete.
\end{proof}

\subsection{Case II: Polynomial near the origin}
\label{sec:pno}

Fix $\lambda \in [0,\infty)$ and let ${\cal A}(\lambda)$ denote the
set of mixing densities, $h$, for which there exists a constant,
$k_\lambda$, such that
\begin{equation*} 
  \frac{\int_0^\infty \frac{1}{\sqrt{u}} \, e^{ -\frac{s u}{2}} \, h(u)
    \, du}{\int_0^\infty \sqrt{u} \, e^{ -\frac{s u}{2}} \, h(u)
    \, du} \le \lambda s + k_\lambda
\end{equation*}
for every $s \ge 0$.  For each mixing density, $h$, we define
\[
\lambda_h = \inf \big\{ \lambda \in [0,\infty): h \in {\cal
  A}(\lambda) \big \} \;.
\]
If $h$ is not in ${\cal A}(\lambda)$ for any $\lambda \in [0,\infty)$,
then we set $\lambda_h = \infty$.  Here is an example.  Suppose that
$h$ is a $\mbox{Gamma}(\alpha,1)$ density.  If $\alpha>1/2$, then
routine calculations show that
\begin{equation}
  \label{eq:gl}
  \frac{\int_0^\infty \frac{1}{\sqrt{u}} \, e^{ -\frac{s u}{2}} \, h(u)
    \, du}{\int_0^\infty \sqrt{u} \, e^{ -\frac{s u}{2}} \, h(u) \, du}
  = \frac{1}{2\alpha-1} s + \frac{2}{2\alpha-1} \;.
\end{equation}
So, in this case, $\lambda_h = \frac{1}{2\alpha-1}$.  On the other
hand, if $\alpha \in (0,1/2]$, then $\lambda_h = \infty$.

Our next result shows that $\lambda_h$ is determined solely by the
behavior of the density $h$ near $0$.

\begin{lemma}
  \label{lem:behnzr}
  Suppose that $h$ and $\tilde{h}$ are two mixing densities that are
  both strictly positive in a neighborhood of zero.  If
\[
\lim_{u \rightarrow 0} \frac{h(u)}{\tilde{h}(u)} \in (0,\infty) \;,
\]
then, $\lambda_h = \lambda_{\tilde{h}}$.
\end{lemma}

\begin{proof}
  Assume that $\lambda_{\tilde{h}} < \infty$.  We will show that
  $\lambda_h \leq \lambda_{\tilde{h}}$.  Fix $\lambda \in
  (\lambda_{\tilde{h}}, \infty)$ arbitrarily.  Let $\lambda^* =
  (\lambda_{\tilde{h}} + \lambda)/2$.  Since $\lim_{u \rightarrow 0}
  \frac{h(u)}{\tilde{h}(u)} \in (0, \infty)$, there exists $\eta > 0$
  such that
\begin{equation}
  \label{eq:zrbhvr1}
  C_{1, \eta} < \frac{h(u)}{\tilde{h}(u)} < C_{2, \eta} 
\end{equation}
for every $u \in (0, \eta]$, where $C_{1, \eta}, C_{2, \eta} \in
\mathbb{R}_+$ satisfy $\frac{C_{2, \eta}}{C_{1, \eta}} =
\sqrt{\frac{\lambda}{\lambda^*}} > 1$.  Also, note that for such an
$\eta$,
\begin{equation*}
  \frac{\int_\eta^\infty \sqrt{u} \, e^{ -\frac{s u}{2}} \, \tilde{h}(u)
    \, du}{\int_{\mathbb{R_+}} \sqrt{u} \, e^{ -\frac{s u}{2}} \,
    \tilde{h}(u) \, du} \le \frac{e^{ -\frac{s \eta}{2}}
    \int_\eta^\infty \sqrt{u} \, \tilde{h}(u) \, du}{\int_0^{\eta/2}
    \sqrt{u} \, e^{ -\frac{s u}{2}} \, \tilde{h}(u) \, du} \le \frac{e^{
      -\frac{s \eta}{4}} \int_\eta^\infty \sqrt{u} \, \tilde{h}(u) \,
    du}{\int_0^{\eta/2} \sqrt{u} \, \tilde{h}(u) \, du} \;.
\end{equation*}
Consequently,
\[
\frac{\int_\eta^\infty \sqrt{u} \, e^{ -\frac{s u}{2}} \, \tilde{h}(u)
  \, du}{\int_{\mathbb{R_+}} \sqrt{u} \, e^{ -\frac{s u}{2}} \,
  \tilde{h}(u) \, du} \rightarrow 0 \;\; \mbox{as $s \rightarrow
  \infty$} \;,
\]
so there exists $s_\eta > 0$ such that
\begin{equation}
  \label{eq:zrbhvr2}
  \frac{\int_0^\eta \sqrt{u} \, e^{ -\frac{s u}{2}} \, \tilde{h}(u) \,
    du}{\int_{\mathbb{R}_+} \sqrt{u} \, e^{ -\frac{s u}{2}} \,
    \tilde{h}(u) \, du} = 1 - \frac{\int_\eta^\infty \sqrt{u} \, e^{
      -\frac{s u}{2}} \, \tilde{h}(u) \, du}{\int_{\mathbb{R}_+}
    \sqrt{u} \, e^{ -\frac{s u}{2}} \, \tilde{h}(u) \, du} \geq
  \sqrt{\frac{\lambda^*}{\lambda}}
\end{equation}
for every $s \geq s_\eta$.  It follows from \eqref{eq:zrbhvr1} and
\eqref{eq:zrbhvr2} that for every $s \geq s_\eta$,
\begin{align*}
  \frac{\int_{\mathbb{R_+}} \frac{1}{\sqrt{u}} \, e^{ -\frac{s u}{2}}
    \, h(u) \, du}{\int_{\mathbb{R_+}} \sqrt{u} \, e^{ -\frac{s u}{2}}
    \, h(u) \, du} & = \frac{\int_0^\eta \frac{1}{\sqrt{u}} \, e^{
      -\frac{s u}{2}} \, h(u) \, du}{\int_{\mathbb{R_+}} \sqrt{u} \,
    e^{ -\frac{s u}{2}} \, h(u) \, du} + \frac{\int_\eta^\infty
    \frac{1}{\sqrt{u}} \, e^{ -\frac{s u}{2}} \, h(u) \,
    du}{\int_{\mathbb{R_+}} \sqrt{u} \, e^{ -\frac{s u}{2}} \, h(u) \,
    du} \\ & \le \frac{\int_0^\eta \frac{1}{\sqrt{u}} \, e^{ -\frac{s
        u}{2}} \, h(u) \, du}{\int_0^\eta \sqrt{u} \, e^{ -\frac{s
        u}{2}} \, h(u) \, du} + \frac{1}{\eta} \frac{\int_\eta^\infty
    \sqrt{u} \, e^{ -\frac{s u}{2}} \, h(u) \, du}{\int_{\mathbb{R_+}}
    \sqrt{u} \, e^{ -\frac{s u}{2}} \, h(u) \, du} \\ & \le
  \frac{C_{2,\eta}}{C_{1,\eta}} \frac{\int_0^\eta \frac{1}{\sqrt{u}}
    \, e^{ -\frac{s u}{2}} \, \tilde{h}(u) \, du}{\int_0^\eta \sqrt{u}
    \, e^{ -\frac{s u}{2}} \, \tilde{h}(u) \, du} + \frac{1}{\eta} \\
  & \le \sqrt{\frac{\lambda}{\lambda^*}}
  \sqrt{\frac{\lambda}{\lambda^*}} \frac{\int_0^\eta
    \frac{1}{\sqrt{u}} \, e^{ -\frac{s u}{2}} \, \tilde{h}(u) \,
    du}{\int_{\mathbb{R}_+} \sqrt{u} \, e^{ -\frac{s u}{2}} \,
    \tilde{h}(u) \, du} + \frac{1}{\eta} \\ & \le
  \frac{\lambda}{\lambda^*} \frac{\int_{\mathbb{R}_+}
    \frac{1}{\sqrt{u}} \, e^{ -\frac{s u}{2}} \, \tilde{h}(u) \,
    du}{\int_{\mathbb{R}_+} \sqrt{u} \, e^{ -\frac{s u}{2}} \,
    \tilde{h}(u) \, du} + \frac{1}{\eta} \;.
\end{align*}
Since $\tilde{h} \in \mathcal{A}(\lambda^*)$, there exists $k$ such
that
\begin{equation}
  \label{eq:zrbhvr3}
  \frac{\int_{\mathbb{R_+}} \frac{1}{\sqrt{u}} \, e^{ -\frac{s u}{2}}
    \, h(u) \, du}{\int_{\mathbb{R_+}} \sqrt{u} \, e^{ -\frac{s u}{2}}
    \, h(u) \, du} \leq \frac{\lambda}{\lambda^*} (\lambda^* s + k) +
  \frac{1}{\eta} = \lambda s + \frac{\lambda}{\lambda^*} k +
  \frac{1}{\eta}
\end{equation}
for every $s \geq s_\eta$.  Our assumptions imply that
$\int_{\mathbb{R_+}} \frac{1}{\sqrt{u}} \, \tilde{h}(u) \, du <
\infty$.  Together with \eqref{eq:zrbhvr1}, this leads to
$\int_{\mathbb{R_+}} \frac{1}{\sqrt{u}} \, h(u) \, du < \infty$.
Then, since
\[
\sup_{s \in (0, s_\eta)} \frac{\int_{\mathbb{R_+}} \frac{1}{\sqrt{u}}
  \, e^{ -\frac{s u}{2}} \, h(u) \, du}{\int_{\mathbb{R_+}} \sqrt{u}
  \, e^{ -\frac{s u}{2}} \, h(u) \, du} \le \sup_{s \in (0, s_\eta)}
\frac{\int_{\mathbb{R_+}} \frac{1}{\sqrt{u}} \, h(u) \, du}{e^{
    -\frac{s}{2}} \int_0^1 \sqrt{u} \, h(u) \, du} \le
\frac{e^{\frac{s_\eta}{2}} \int_{\mathbb{R_+}} \frac{1}{\sqrt{u}} \,
  h(u) \, du}{\int_0^1 \sqrt{u} \, h(u) \, du} \;,
\]
it follows from \eqref{eq:zrbhvr3} that $h \in \mathcal{A}(\lambda)$.
Hence, $\lambda_h \le \lambda$.  Since $\lambda \in
(\lambda_{\tilde{h}}, \infty)$ was arbitrarily chosen, it follows that
$\lambda_h \le \lambda_{\tilde{h}}$.

Now assume that $\lambda_h < \infty$.  We can show that
$\lambda_{\tilde{h}} \le \lambda_h$ by noting that
\[
\lim_{u \rightarrow 0} \frac{h(u)}{\tilde{h}(u)} \in (0, \infty)
\Leftrightarrow \lim_{u \rightarrow 0} \frac{\tilde{h}(u)}{h(u)} \in
(0,\infty) \;,
\]
and reversing the roles of $h$ and $\tilde{h}$ in the above argument.
We have shown that $\lambda_h < \infty$ if and only if
$\lambda_{\tilde{h}} < \infty$, and when they are finite, they are
equal.
\end{proof}

\begin{corollary}
  \label{prop:pno}
  Let $h$ be a mixing density that satisfies condition ${\cal M}$.  If
  $h$ is polynomial near the origin with power $c >
  \frac{n-p+2a-d-1}{2}$, then the posterior distribution is proper and
  the DA Markov chain is geometrically ergodic.
\end{corollary}

\begin{proof}
  We can write
\begin{equation}
  \label{eq:hstar}
  \frac{\int_0^\infty u^{\frac{d-2}{2}} \, e^{ -\frac{s u}{2}} \, h(u)
    \, du}{\int_0^\infty u^{\frac{d}{2}} \, e^{ -\frac{s u}{2}} \, h(u)
    \, du} = \frac{\int_0^\infty \frac{1}{\sqrt{u}} \, e^{ -\frac{s
        u}{2}} \, h^*(u) \, du}{\int_0^\infty \sqrt{u} \, e^{ -\frac{s
        u}{2}} \, h^*(u) \, du} \;,
\end{equation}
where $h^*(u)$ is the mixing density that is proportional to
$u^{\frac{d-1}{2}} h(u)$.  It's easy to see that $h^*$ is polynomial
near the origin with power $c' > \frac{n-p+2a-2}{2}$.  (Note that
$(N2)$ implies that $c' > 0$, so the integral in the numerator on the
right-hand side of \eqref{eq:hstar} is finite.)  Let $\tilde{h}$ be
the $\mbox{Gamma}(c'+1,1)$ density, which is clearly polynomial near
the origin with power $c'$.  Then,
  \[
  \lim_{u \rightarrow 0} \frac{h^*(u)}{\tilde{h}(u)} = \lim_{u
    \rightarrow 0} \frac{h^*(u)}{u^{c'}} \frac{u^{c'}}{\tilde{h}(u)}
  \in (0,\infty) \;.
  \]
  Thus, \eqref{eq:gl} and Lemma~\ref{lem:behnzr} imply that
  $\lambda_{h^*} = \lambda_{\tilde{h}} = 1/(2c'+1)$, and the result now
  follows from Proposition~\ref{prop:drift_smn} since
  \[
  \lambda_{h^*} = \frac{1}{2c'+1} < \frac{1}{n-p+2a-1} \;.
  \]
\end{proof}

\subsection{Case III: Faster than polynomial near the origin}
\label{sec:fpno}

\begin{lemma} 
  \label{lem:incrsrt}
  Suppose that $h$ and $\tilde{h}$ are two mixing densities that are
  both strictly positive in a neighborhood of zero.  If there exists
  $\eta > 0$ such that $\frac{h}{\tilde{h}}$ is a strictly increasing
  function on $(0, \eta]$, then $\lambda_h \leq \lambda_{\tilde{h}}$.
\end{lemma}

\begin{proof}
  First, fix $s>0$ and define two densities as follows: $h_{s,\eta}
  (u) = K_{s,\eta} \, e^{ -\frac{s u}{2}} \, h(u) \, I_{(0,\eta)}(u)$
  and $\tilde{h}_{s,\eta} (u) = \tilde{K}_{s,\eta} \, e^{ -\frac{s
      u}{2}} \, \tilde{h}(u) \, I_{(0,\eta)}(u)$, where $K_{s,\eta}$
  and $\tilde{K}_{s,\eta}$ are normalizing constants.  Since
  $\frac{h}{\tilde{h}}$ is strictly increasing on $(0, \eta]$, it
  follows that
\[
\frac{h_{s,\eta} (u)}{\tilde{h}_{s,\eta} (u)} > 1 \Leftrightarrow
\frac{h(u)}{\tilde{h}(u)} > \frac{\tilde{K}_{s,\eta}}{K_{s,\eta}}
\Leftrightarrow u > u^*
\]
for some $u^* \in (0, \eta)$.  This shows that the densities
$\tilde{h}_{s,\eta}$ and $h_{s,\eta}$ cross exactly once in the
interval $(0, \eta)$, which is their common support.  It follows that
a random variable with density $\tilde{h}_{s,\eta}$ is stochastically
dominated by a random variable with density $h_{s,\eta}$.  This
stochastic dominance implies that
\begin{equation}
  \label{eq:dominance}
  \int_0^\eta \frac{1}{\sqrt{u}} \, \tilde{h}_{s,\eta}(u) \, du \ge
  \int_0^\eta \frac{1}{\sqrt{u}} \, h_{s,\eta}(u) \, du \hspace*{4mm}
  \mbox{and} \hspace*{4mm} \int_0^\eta \sqrt{u} \, \tilde{h}_{s,\eta}(u)
  \, du \le \int_0^\eta \sqrt{u} \, h_{s,\eta}(u) \, du \;.
\end{equation}
Now define two more densities as follows
\[
h_\eta(u) = \frac{h(u)}{\int_0^\eta h(v) dv} I_{(0,\eta)}(u)
\hspace*{4mm} \mbox{and} \hspace*{4mm} \tilde{h}_\eta(u) =
\frac{\tilde{h}(u)}{\int_0^\eta \tilde{h}(v) dv} I_{(0,\eta)}(u) \;.
\]
It follows from \eqref{eq:dominance} that
\[
\frac{\int_{\mathbb{R_+}} \frac{1}{\sqrt{u}} \, e^{ -\frac{s u}{2}} \,
  \tilde{h}_\eta(u) \, du}{\int_{\mathbb{R_+}} \sqrt{u} \, e^{
    -\frac{s u}{2}} \, \tilde{h}_\eta(u) \, du} =
\frac{\int_{\mathbb{R_+}} \frac{1}{\sqrt{u}} \, \tilde{h}_{s,\eta}(u)
  \, du}{\int_{\mathbb{R_+}} \sqrt{u} \, \tilde{h}_{s,\eta}(u) \, du}
\ge \frac{\int_{\mathbb{R_+}} \frac{1}{\sqrt{u}} \, h_{s,\eta}(u) \,
  du}{\int_{\mathbb{R_+}} \sqrt{u} \, h_{s,\eta}(u) \, du} =
\frac{\int_{\mathbb{R_+}} \frac{1}{\sqrt{u}} \, e^{ -\frac{s u}{2}} \,
  h_\eta(u) \, du}{\int_{\mathbb{R_+}} \sqrt{u} \, e^{ -\frac{s u}{2}}
  \, h_\eta(u) \, du} \;.
\]
Hence, $\lambda_{h_\eta} \leq \lambda_{\tilde{h}_\eta}$.  Since
\[
\lim_{u \rightarrow 0} \frac{h(u)}{h_\eta(u)} = \int_0^\eta h(v) \, dv
\in \mathbb{R}_+ \hspace*{4mm} \mbox{and} \hspace*{4mm} \lim_{u
  \rightarrow 0} \frac{\tilde{h}(u)}{\tilde{h}_\eta(u)} = \int_0^\eta
\tilde{h}(v) \, dv \in \mathbb{R}_+ \;,
\]
it follows from Lemma~\ref{lem:behnzr} that $\lambda_h =
\lambda_{h_\eta}$ and $\lambda_{\tilde{h}} =
\lambda_{\tilde{h}_\eta}$.
\end{proof}

\begin{corollary}
  \label{prop:fpno}
  Let $h$ be a mixing density that satisfies condition ${\cal M}$.  If
  $h$ is faster than polynomial near the origin, then the posterior
  distribution is proper and the DA Markov chain is geometrically
  ergodic.
\end{corollary}

\begin{proof}
  Again, define $h^*(u)$ to be the mixing density that is proportional
  to $u^{\frac{d-1}{2}} h(u)$.  In light of \eqref{eq:hstar}, it
  suffices to show that $\lambda_{h^*} = 0$.  First, note that $h^*$
  is faster than polynomial near the origin.  Fix $c>0$ and define
  $\tilde{h}(u) = (c+1)\, u^c \, I_{(0,1)}(u)$.  Clearly,
  $\lambda_{\tilde{h}} = \frac{1}{2c+1}$.  Since $h^*$ is faster than
  polynomial near the origin, there exists $\eta_c \in (0,1)$ such
  that $\frac{h^*(u)}{\tilde{h}(u)}$ is strictly increasing in
  $(0,\eta_c)$.  Thus, Lemma~\ref{lem:incrsrt} implies that
  $\lambda_{h^*} \le \lambda_{\tilde{h}} = \frac{1}{2c+1}$.  But $c$
  was arbitrary, so $\lambda_{h^*} = 0$.  The result now follows
  immediately from Proposition~\ref{prop:drift_smn}.
\end{proof}

Taken together, Corollaries~\ref{prop:zno}, \ref{prop:pno} and
\ref{prop:fpno} are equivalent to Theorem~\ref{thm:main}.  Hence, our
proof of Theorem~\ref{thm:main} is complete.

\section{Examples and a result concerning mixtures of mixing
  densities}
\label{sec:ex}

We claimed in the Introduction that every mixing density which is a
member of a standard parametric family is either polynomial near the
origin, or faster than polynomial near the origin.  Here we provide
some details.  When we write $W \sim \mbox{Gamma}(\alpha, \gamma)$, we
mean that $W$ has density proportional to $w^{\alpha-1} e^{-w \gamma}
I_{\mathbb{R}_+}(w)$.  By $W \sim \mbox{Beta}(\alpha, \gamma)$, we
mean that the density is proportional to $w^{\alpha-1}
(1-w)^{\gamma-1} I_{(0,1)}(w)$, and by $W \sim
\mbox{Weibull}(\alpha,\gamma)$, we mean that the density is
proportional to $w^{\alpha-1} e^{-\gamma w^\alpha}
I_{\mathbb{R}_+}(w)$.  In all three cases, we need $\alpha,\gamma>0$.
It is clear that these densities are all polynomial near the origin
with $c=\alpha-1$.  Moreover, condition ${\cal M}$ always holds.
Hence, according to Theorem~\ref{thm:main}, if the mixing density is
$\mbox{Gamma}(\alpha, \gamma)$, $\mbox{Beta}(\alpha, \gamma)$ or
$\mbox{Weibull}(\alpha,\gamma)$ with $\alpha > \frac{n-p+2a-d+1}{2}$,
then the DA Markov chain is geometrically ergodic.

By $W \sim \mbox{F}(\nu_1,\nu_2)$, we mean that $W$ has density
proportional to
\[
\frac{w^{(\nu_1-2)/2}}{\big(1 + (\frac{\nu_1}{\nu_2}) w
  \big)^{(\nu_1+\nu_2)/2}} \, I_{\mathbb{R}_+}(w) \;,
\]
where $\nu_1,\nu_2>0$.  These densities are polynomial near the origin
with $c=(\nu_1-2)/2$.  To get a geometric chain in this case, we need
$\nu_1 > n-p+2a-d+1$ and $\nu_2>d$.  (The second condition is to
ensure that condition ${\cal M}$ holds.)  Consider the shifted Pareto
family with density given by
\[
\frac{\gamma \alpha^\gamma}{(w+\alpha)^{\gamma+1}} \,
I_{\mathbb{R}_+}(w) \;,
\]
where $\alpha,\gamma>0$.  This density is polynomial near the origin
with $c=0$.  Since the requirement that $c > \frac{n-p+2a-d-1}{2}$
forces $c$ to be strictly positive, Theorem 1 is not applicable to
this family.

By $W \sim \mbox{IG}(\alpha, \gamma)$, we mean that $W$ has density
proportional to $w^{-\alpha-1} e^{-\gamma/w} I_{\mathbb{R}_+}(w)$,
where $\alpha,\gamma>0$.  For any $c>0$, the derivative of
$\log(h(w)/w^c)$ is
\[
\frac{-(\alpha+c+1)}{w} + \frac{\gamma}{w^2} = \frac{1}{w} \bigg[
-(\alpha+c+1) + \frac{\gamma}{w} \bigg] \;,
\]
which is clearly strictly positive in a neighborhood of zero.  Hence,
the $\mbox{IG}(\alpha,\gamma)$ densities are all faster than
polynomial near the origin.  Thus, Theorem~\ref{thm:main} implies
that, as long as $\alpha>d/2$, the DA Markov chain is geometrically
ergodic.

By $W \sim \mbox{GIG}(v,a,b)$, we mean that $W$ has a generalized
inverse Gaussian distribution with density given by
\begin{equation*}
  h(w) = \frac{1}{2 K_v \big( \sqrt{ab} \big)} \Big( \frac{a}{b}
  \Big)^{\frac{v}{2}} w^{v-1} \exp \Big \{ - \frac{1}{2} \Big(aw +
  \frac{b}{w} \Big) \Big \} I_{\mathbb{R}_+}(w) \;,
\end{equation*}
where $a,b \in \mathbb{R}_+$ and $v \in \mathbb{R}$.  Taking $v =
-\frac{1}{2}$ leads to the standard inverse Gaussian density (with a
nonstandard parametrization).  By $W \sim
\mbox{Log-normal}(\mu,\gamma)$, we mean that $W$ has density
proportional to
\begin{equation*}
  \frac{1}{w} \exp \Big \{ -\frac{1}{2 \gamma} \Big( \log w - \mu)^2 \Big)
  \Big \} I_{\mathbb{R}_+}(w) \;,
\end{equation*}
where $\mu \in \mathbb{R}$ and $\gamma>0$.  By $W \sim
\mbox{Fr\'{e}chet}(\alpha,\gamma)$, we mean that $W$ has density
proportional to
\begin{equation*}
  w^{-(\alpha+1)} \, e^{-\frac{\gamma^\alpha}{w^\alpha}} \, 
  I_{\mathbb{R}_+}(w) \;,
\end{equation*}
where $\alpha,\gamma>0$.  Arguments similar to those used in the
inverted gamma case above show that all members of these three
families are faster than polynomial near the origin.  Moreover,
condition ${\cal M}$ holds for all the Log-normal and GIG densities,
and for all $\mbox{Fr\'{e}chet}(\alpha,\gamma)$ densities with
$\alpha>d/2$.  Thus, the corresponding DA Markov chains are all
geometric.

We end this section with a result concerning mixtures of mixing
densities.

\begin{proposition}
  \label{prop:mix}
  Let $I$ be an index set equipped with a probability measure
  $\xi$. Consider a family of mixing densities $\{h_a\}_{a \in I}$
  such that $\lambda_{h_a} = 0$ for every $a \in I$.  In particular,
  for every $a \in I$ and every $\lambda \in (0,1)$, there exists
  $k_{a, \lambda} > 0$ such that
\[
\frac{\int_0^\infty \frac{1}{\sqrt{u}} \, e^{ -\frac{s u}{2}} \,
  h_a(u) \, du}{\int_0^\infty \sqrt{u} \, e^{ -\frac{s u}{2}} \,
  h_a(u) \, du} \le \lambda s + k_{a,\lambda}
\]
for every $s \ge 0$.  Suppose that, for every $\lambda \in (0,1)$,
\begin{equation}
  \label{eq:sup}
\sup_{a \in I} k_{a, \lambda} < \infty \;.
\end{equation}
Then $\lambda_h = 0$ where $h(u) = \int_I h_a(u) \, \xi(da)$.
\end{proposition}

\begin{proof}
  Fix $\lambda \in (0,1)$.  For every $s \ge 0$, we have
\begin{align*}
  \frac{\int_0^\infty \frac{1}{\sqrt{u}} \, e^{ -\frac{s u}{2}} \,
    h(u) \, du}{\int_0^\infty \sqrt{u} \, e^{ -\frac{s u}{2}} \, h(u)
    \, du} & = \frac{\int_I \big( \int_0^\infty \frac{1}{\sqrt{u}} \,
    e^{ -\frac{s u}{2}} \, h_a(u) \, du \big) \xi(da)}{\int_0^\infty
    \sqrt{u} \, e^{ -\frac{s u}{2}} \, h(u) \, du} \\ & \le
  \frac{\int_I (\lambda s + k_{a,\lambda}) \big( \int_0^\infty
    \sqrt{u} \, e^{ -\frac{s u}{2}} \, h_a(u) \, du \big)
    \xi(da)}{\int_0^\infty \sqrt{u} \, e^{ -\frac{s u}{2}} \, h(u) \,
    du} \\ & \le \big( \lambda s + \sup_{a \in I} k_{a, \lambda} \big)
  \frac{\int_I \int_0^\infty \sqrt{u} \, e^{ -\frac{s u}{2}} \, h_a(u)
    \, du \, \xi(da)}{\int_0^\infty \sqrt{u} \, e^{ -\frac{s u}{2}} \,
    h(u) \, du} \\ & = \lambda s + \sup_{a \in I} k_{a, \lambda} \;.
\end{align*}
Since this holds for all $\lambda \in (0,1)$, the result follows.
\end{proof}

\begin{remark}
  If the index set, $I$, in Proposition~\ref{prop:mix} is a finite
  set, then \eqref{eq:sup} is automatically satisfied.
\end{remark}

Here's a simple application of Proposition~\ref{prop:mix}.

\begin{proposition}
  \label{prop:mix2}
  Let $\{h_i\}_{i=1}^M$ be a finite set of mixing densities that all
  satisfy condition ${\cal M}$, and are all either zero near the
  origin, or faster than polynomial near the origin.  Define
\[
h(u) = \sum_{i=1}^M w_i \, h_i(u) \;,
\]
where $w_i > 0$ and $\sum_{i=1}^M w_i = 1$.  Then the posterior
distribution is proper and the DA Markov chain is geometrically
ergodic.
\end{proposition}

\begin{proof}
  Since Proposition~\ref{prop:mix} implies that $\lambda_h=0$, the
  arguments in the proof of Corollary~\ref{prop:fpno} can be applied
  to prove the result.
\end{proof}

\vspace*{5mm}

\noindent {\bf \large Acknowledgment}.  The first author was supported
by NSF Grants DMS-11-06395 \& DMS-15-11945, and the third by NSF Grants
DMS-11-06084 \& DMS-15-11945.

\vspace*{8mm}

\noindent {\Large \bf Appendix: Matrix Normal and Inverse Wishart Densities}
\begin{appendix}

\begin{description}
\item[Matrix Normal Distribution] Suppose $Z$ is an $r \times c$
  random matrix with density
\[
f_{Z}(z) = \frac{1}{(2\pi)^{\frac{rc}{2}} |A|^{\frac{c}{2}}
  |B|^{\frac{r}{2}}} \exp \bigg[ -\frac{1}{2}\mbox{tr} \Big\{ A^{-1}(z
  - \theta) B^{-1} (z - \theta)^T \Big\} \bigg] \;,
\]
where $\theta$ is an $r \times c$ matrix, $A$ and $B$ are $r \times r$
and $c \times c$ positive definite matrices.  Then $Z$ is said to have
a \textit{matrix normal distribution} and we denote this by $Z \sim
\mbox{N}_{r,c} (\theta,A,B)$ \citep[][Chapter 17]{arno:1981}.

\item[Inverse Wishart Distribution] Suppose $W$ is an $r \times r$
  random positive definite matrix with density
\[
f_{W}(w) = \frac{|w|^{-\frac{m+r+1}{2}} \exp \Big \{ -\frac{1}{2}
  \mbox{tr} \big( \Theta^{-1} w^{-1} \big) \Big\}}{ 2^{\frac{mr}{2}}
  \pi^{\frac{r(r-1)}{4}} |\Theta|^{\frac{m}{2}} \prod_{i=1}^r \Gamma
  \big( \frac{1}{2}(m+1-i) \big)} I_{{\cal S}_r}(W) \;,
\]
where $m > r-1$ and $\Theta$ is an $r \times r$ positive definite
matrix.  Then $W$ is said to have an \textit{inverse Wishart
  distribution} and this is denoted by $W \sim \mbox{IW}_r(m,
\Theta)$.
\end{description}
\end{appendix}

\bibliographystyle{ims}
\bibliography{refs}

\end{document}